\documentclass[11pt]{amsart}
\usepackage{amssymb} 

\input xy
\xyoption{all}

\parskip=\smallskipamount

\newtheorem{theorem}{Theorem}[section]

\newtheorem{corollary}[theorem]{Corollary}
\newtheorem{proposition}[theorem]{Proposition}

\theoremstyle{definition}

\newtheorem{problem}[theorem]{Problem}

\newcommand{\C}{\mathbb{C}}

\newcommand{\Z}{\mathbb{Z}}
\renewcommand{\P}{\mathbb{P}}

\newcommand{\cC}{\mathcal{C}}

\newcommand{\cO}{\mathcal{O}}

\newcommand{\fo}{\mathcal{O}_{f.o.}}

\renewcommand{\d}{\mathrm{d}}

\def\bs{\backslash}

\begin{document}
\title{Gunning-Narasimhan's theorem with a growth condition}
\author{Franc Forstneri\v c and Takeo Ohsawa}
\address{
Faculty of Mathematics and Physics, University of Ljubljana, 
and Institute of Mathematics, Physics and Mechanics, Jadranska 19, 
1000 Ljubljana, Slovenia}
\email{franc.forstneric@fmf.uni-lj.si}

\address{
Graduate School of Mathematics, 
Nagoya University, 
Chikusaku Furocho, 
464-8602 Nagoya, Japan}
\email{ohsawa@math.nagoya-u.ac.jp}
\thanks{The research of the first named author was supported 
by grants P1-0291 and J1-2152 from ARRS, Republic of Slovenia.}

%
%
%
%
\subjclass[2000]{32E10, 32E30, 32H02; 14H05}
\date{\today}
\keywords{Riemann surface, divisor, function of finite order}

%
%
%
%
\begin{abstract}
Given a compact Riemann surface $X$ and a point $x_0\in X$,
we construct a holomorphic function without critical points 
on the punctured Riemann surface $R=X\bs \{x_0\}$
which is of finite order at $x_0$.
\end{abstract}
\maketitle

\section{The statement}
Let $X$ be a compact Riemann surface, let $x_0$ be an arbitrary point
of $X$, and let $R = X\bs \{x_0\}$. The set of holomophic functions
on $R$ will be denoted by $\cO(R)$. Let $U\subset X$ be a coordinate
neighborhood of the point $x_0$ and let $z$ be a local coordinate 
on $U$ with $z(x_0)=0$.
A holomorphic function $f\in \cO(R)$ on $R$ 
is said to be {\em of finite order}
(at the point $x_0$) if there exist positive numbers 
$\lambda$ and $\mu$ such that 
\begin{equation}
\label{eq:fo}
	|f(z)| \le \lambda \exp |z|^{-\mu} \quad\hbox{holds on}\ U\bs\{x_0\}.
\end{equation}
We denote by $\fo(R)$ the set of all holomorphic 
functions of finite order on $R$. For any $f\in \fo(R)$,
the {\em order} of $f$ is defined as the infimum of all
numbers $\mu>0$ such that (\ref{eq:fo}) holds for some 
$\lambda>0$.
By using Poisson-Jensen's formula it is easy to see that, 
for any nonvanishing holomorphic function
$f$ on $U\bs \{x_0\}$ satisfying (\ref{eq:fo}), there 
exist a neighborhood $V\ni x_0$ and a number $\chi>0$
such that 
$\frac{1}{|f(z)|}  \le  \chi \exp |z|^{-\mu}$ 
on $V \bs\{x_0\}$ (Hadamard's theorem, c.f.\ 
\cite[Chap.\ 5]{Ahlfors1966}).

In 1967 Gunning and Narasimhan proved 
that every open Riemann surface admits a holomorphic function 
without critical points \cite{Gunning-Narasimhan}. 
Our goal is to prove the following result
for punctured Riemann surfaces.

\begin{theorem}
\label{T1}     
If $X$ is a compact Riemann surface and $x_0\in X$ then
the punctured Riemannn surface $R=X\bs \{x_0\}$ admits 
a noncritical holomorphic function of finite order; that is,
$\{f \in \fo(R)\colon \d f \ne 0\ \hbox{everywhere}\}\ne \emptyset$.  
\end{theorem}

We show that this result is the best possible one, except when 
$X=\C\P^1$ is the Riemann sphere in which case $R=\C$:

\begin{proposition}
If $X$ is a compact Riemann surface of genus $g\ge 1$
and $x_0\in X$ then every algebraic function 
$X\bs \{x_0\} \to\C$ has a critical point.
\end{proposition}

In the case when $X$ is a torus, this was shown in \cite[\S 4]{Majcen}.

\begin{proof}
Assume that $f\colon R=X\bs\{x_0\}\to \C$ is an algebraic function.
Then $f$ extends to a meromorphic map $X\to\C\P^1$
sending $x_0$ to the point $\infty$. Let $d$ denote the degree 
of $f$ at $x_0$, so $f$ equals  the map
$z\mapsto z^d$ in a certain pair of local holomorphic
coordinates at the points $x_0$ and $\infty$.
Since $f^{-1}(\infty)=\{x_0\}$, $d$ is also the global degree of $f$.  
By the Riemann-Hurwitz formula (see \cite{Hartshorne})
we then have
\[
	\chi(X)=d\chi(\C\P^1) - b,
\]
where $\chi(X)$ is the Euler
number of $X$ and $b$ is the total branching order of $f$
(the sum of its local branching orders over the points of $X$).
If we assume that $f$ has no critical points on $R$,
then it only branches at $x_0$, and its branching order 
at $x_0$ is clearly $b=d-1$. Hence the above equation reads
$2-2g = 2d - (d-1)= d+1\ge 1$ which is clearly impossible  
if $g\ge 1$. In fact, we see that any algebraic function 
$f\colon R=X\bs \{x_0\}\to \C$ with degree $d$ at $x_0$ 
must have precisely $(d+1)-(2-2g)= d+2g-1$ branch points
in $R$ when counted with algebraic multiplicities.
\end{proof}

\section{Preliminaries}
\label{Prel}
We assume that $X$ and $R=X\bs \{x_0\}$ are as above.

\begin{proposition}
\label{P1}
For any effective divisor  $\delta$  on $X$ whose support 
does not contain the point $x_0$ there exists $f\in \fo(R)$  
whose zero divisor $f^{-1}(0)$ coincides with $\delta$.   
\end{proposition}

\begin{proof}
Since holomorphic vector bundles over noncompact 
Riemann surfaces are trivial by Grauert's Oka principle, 
there exists a holomorphic function  
$f_0$ on  $R$  whose zero divisor equals $\delta$. 
Let  $V$  be a disc neighborhood of the point $x_0$ in $X$,
with a holomorphic coordinate $z$ in which $z(x_0)=0$,  
such that $f_0$ does not vanish on $V \bs\{x_0\}$. 
Let $m\in \Z$ denote the winding number of $f$ around 
the point $x_0$. Choose a meromorphic function $h$ on $X$ 
such that $h(z)= c(z) z^m$ on $z\in V$ for some nonvanishing
holomorphic function $c$ on $V$, and such that 
all remaining zeros and poles of $h$ lie in $X\bs \overline V$.
Then $f_0/h$ is a nowhere vanishing holomorphic function 
with winding number zero in $V\bs \{x_0\}$, and hence
$\log(f_0/h)$  has a single valued holomorphic branch on $V\bs \{x_0\}$. 
Choose a smaller disc $W\Subset V$ centered at $x_0$.
By solving a Cousin-I problem we find a 
holomorphic functions $u_1$ on $X\bs \overline W$  
and $u_2$  on $V\bs \{x_0\}$ such that  
$u_1 - u_2 =\log(f_0/h)$ holds on $V \bs \overline W$,
and such that $x_0$ is a pole of the function $u_2$. 
Hence, letting  $f = h e^{-u_2}$  on  $V\bs\{x_0\}$ 
and  $f = f_0 e^{-u_1}$ on $R\bs W$, we obtain  
a function $f\in\fo(R)$ satisfying $f^{-1}(0) = \delta$.  
\end{proof}

Let  $L\to X$  be a holomorphic line bundle and let  $h$  
be a fiber metric of $L$.  A holomorphic section $s$  of  $L$ 
over  $R$  is said to be of finite order if the length  
$|s|$ of $s$ with respect to $h$ satisfies on $U\bs \{x_0\}$,
as a function of the local coordinate $z$, that          
\[
	|s|(z) \le \lambda \exp |z|^{-\mu} 
	\quad\hbox{for\ some}\ \lambda,\mu\in (0,\infty).
\]	
The order of $s$ is defined similarly as in the case of 
holomorphic functions. Since every holomorphic line bundle 
over $X$ is associated with a divisor, Proposition \ref{P1}
implies the following:

\begin{proposition}
\label{P2}    
For any holomorphic line bundle $L$ over  $X$, there exists 
a holomorphic section $s$ of the restricted bundle $L|_R$  
such that $s$ is of finite order and $s(x)\ne 0$
for all $x\in R$. 
\end{proposition}

\begin{proof}
Let $v$ be any meromorphic nonzero section of $L$. 
Let $p_1,\ldots,p_m$ (resp.\ $q_1,\ldots,q_n$) be the poles 
(resp.\ the zeros) of $v$ in $R$. By Proposition \ref{P1} 
there exist functions $f,g\in \fo(R)$  such that 
$p_1+p_2+\ldots+p_m$ (resp.\ $q_1+ \ldots +q_n$) is 
the zero divisor of  $f$ (resp.\ of $g$). 
Then the section $s = fv/g$ satisfies the stated properties.
\end{proof}

\begin{corollary}
\label{C1}
There exists a holomorphic 1-form of finite order on $R$ which does not vanish anywhere. 
\end{corollary}

Let $\omega$ be a nowhere vanishing holomorphic 1-form of finite 
order on $R$ guaranteed by Corollary \ref{C1}. 
Then, Theorem \ref{T1} is equivalent to saying that there exists  
a function $g\in\fo(R)$ such that $g^{-1}(0) =\emptyset$ 
and $\int_\gamma g\omega = 0$ holds for any 1-cycle $\gamma$ on $R$, 
for the primitives of $g\omega$ will then be without critical points 
and clearly of finite order, the converse being obvious.     

We shall show that such $g$ can be found in a subset of 
$\fo(R)$ consisting of functions of the form  
$\exp \int_{x_1}^x \eta$ where $\eta$ are meromorphic 1-forms 
on $X$ which are holomorphic on $R$ and $x_1\in R$ is an
arbitrary fixed point in $R$.

Let us denote by $\Omega^1_{alg}(R)$ (resp.\ $\cO_{alg}(R)$) 
the set of meromorphic 1-forms (resp.\ meromorphic functions) on  
$X$ which are holomorphic on $R$. The general 
theory of coherent algebraic sheaves on affine algebraic varieties 
implies the following (c.f.\ \cite{Serre} or \cite{Hartshorne}).

\begin{proposition}
\label{P3} 
Every element of $H^1(R;\C)$ is represented by an element 
of $\Omega^1_{alg}(R)$ as a de Rham cohomology class.      
\end{proposition}

Let $K$ be a compact set in $R$ and let $\cO(K)$  denote the set 
of all continuous functions on $K$ 
which are holomorphically extendible to 
some open neighborhoods of $K$ in $X$. 
Then the Runge approximation theorem says the 
following in our situation.

\begin{proposition}
\label{P4}
For any compact set  $K\subset R$  such that $R\bs K$ is connected, 
the image of the restriction map $\cO_{alg}(R)\to \cO(K)$ is dense 
with respect to the topology of uniform convergence.
\end{proposition}

The proof of Theorem \ref{T1} to be given below is basically a 
combination of Corollary \ref{C1}, Proposition \ref{P3} 
and Proposition \ref{P4}. In order to make a  
short cut argument, we shall apply a refined version 
of Proposition \ref{P4} (Mergelyan's theorem) below.

\section{Proof of Theorem \ref{T1}}      
\label{Proof-T1}
For any $\cC^1$ curve $\alpha\colon [0,1]\to X$  
we denote by $|\alpha|$ its trace, i.e.,
$|\alpha| = \{\alpha(t)\colon 0 \le t \le 1\}$.
If $\alpha$ is closed ($\alpha(0)=\alpha(1)$), we denote 
by $[\alpha]$ its homology class in $H_1(X;\Z)$. 

Let $g$ denote the genus of $X$. There exist 
simple closed real-analytic curves 
$\alpha_1,\ldots,\alpha_{2g}$ in $R$ 
satisfying 
\begin{equation}
\label{eq:hom}
	H_1(X;\Z) = \sum_{i=1}^{2g} \Z [\alpha_i]
\end{equation}
such that $\cap_{i=1}^{2g} |\alpha_i| = \{p\}$ holds for some 
point $p\in R$ and such that, putting 
$\Gamma=\cup_{i=1}^{2g} |\alpha_i|$, the complement     
$R\bs \Gamma$ is connected.

Let $\omega$ nowhere vanishing holomorphic 1-form of finite 
order on $R$ furnished by Corollary \ref{C1}. 
For each curve $\alpha_i$ there is a neighborhood 
$U_i \supset |\alpha_i|$ in $R$ and a 
biholomorphic map $\varphi_i$ from an annulus 
$A_r=\{w \in \C \colon 1-r < |w| < 1+r\}$ onto $U_i$ 
for a sufficiently small $r>0$ such that the 
positively oriented unit circle 
$\{|w|=1\}$ is mapped by $\varphi_i$ onto the curve 
$\alpha_i$, with $\varphi_i(1)=p$.
For $i=1,\ldots,2g$ put
\[
	 \phi_i^* \omega= H_i(w)\, \d w, \qquad     
	 n_i = \frac{1}{2\pi \sqrt{-1}} \int_{|w|=1}\d \log H_i \in \Z. 
\]
By Proposition \ref{P3} there exists 
$\xi\in \Omega^1_{alg}(R)$ such that  
\begin{equation}
\label{eqn:ni}
        n_i   =  \frac{1}{2\pi \sqrt{-1}} \int_{\alpha_i} \xi,
		\qquad i=1,\ldots, 2g.
\end{equation}        
Let  
\[
		u(x) = \exp \int_{p}^x \xi,\qquad x\in R.
\]
By (\ref{eq:hom}) and (\ref{eqn:ni}) the 
integral is independent of the path
in $R$.  (Note that the cycle around the deleted point $x_0$ is 
homologous to zero in $R$.) Hence the function $u$ is 
well defined, single-valued and nonvanishing on $R$, 
and $u\in\fo(R)$ because $\xi\in \Omega^1_{alg}(R)$.     
Replacing $\omega$ by $\omega/u$ we obtain a nowhere vanishing 
1-form of finite order on $R$, 
still denoted $\omega$, 
for which the winding numbers $n_i$ in (\ref{eqn:ni}) equal zero. 
It follows that for every $i=1,\ldots,2g$ we have
\[
	\varphi_i^*\omega = e^{h_i(w)+c_i} \,\d w
\]
for some constants $c_i\in \C$ and holomorphic function $h_i$ 
on the annulus $A_r\subset \C$ with $h_i(1)=0$.
Note that the functions
$h_i\circ\varphi_i^{-1} \colon |\alpha_i|\to\C$ agree at the
unique intersection point $p$ of the curves $|\alpha_i|$, 
and hence they define a continuous function $H$ 
on $\Gamma=\cup_i |\alpha_i|$. 
For every $h\in \cO_{alg}(R)$ we have
\[
	\int_{\alpha_i} e^{-h}\omega = 
	e^{c_i} 
	\int_{|w|=1} e^{h_i - h\circ \varphi_i} \, \d w. 
\]
These numbers can be made arbitrarily small by  choosing 
$h$ to approximate $H$ uniformly on $\Gamma$
(which is equivalent to asking that $h_i-h\circ\varphi_i$
is small on $\{|w|=1\}$ for every $i=1,\ldots,2g$).
Such $h$ exist by Mergelyan's theorem: 
Since $R\bs \Gamma$ is connected, 
every continuous function on $\Gamma$ is a uniform limit
of functions in $\cO_{alg}(R)$ (c.f.\ \cite[Chap.\ 3]{Gaier}).

We assert that there exist functions $f_i\in \cO_{alg}(R)$ 
for $i=1,\ldots,2g$ and a number $\epsilon>0$ 
such that, for any $h\in \cO_{alg}(R)$ satisfying 
\begin{equation}
\label{eq:4}      
		\sup_{|\alpha_i|} |h_i\circ\varphi_i^{-1} - h| < \epsilon,  
		\qquad i=1,\ldots, 2g,
\end{equation}
there exist numbers $\zeta_i\in\C$ ($i=1,\ldots,2g$) such that
\begin{equation}
\label{eq:5} 
		\int_{\alpha_j} 
		\exp\left( \sum_{i=1}^{2g} \zeta_i f_i - h\right) \omega =0,
		\qquad   j=1,\ldots,2g.
\end{equation}
To prove this assertion, which clearly implies Theorem \ref{T1}
(the potential of the 1-form under the integral in 
(\ref{eq:5}) is a holomorphic function of finite order
and without critical points on $R$), 
choose functions $f_i\in \cO_{alg}(R)$ for $i=1,\ldots,2g$ satisfying  
\begin{equation}
\label{eq:6} 
   e^{c_j} \int_{|w|=1} 
   f_i \circ\varphi_j(w)\, \d w = \delta_{ij},  
\end{equation}
where $\delta_{ij}$ denotes the Kronecker's delta. 
Such $f_i$ exist by Proposition \ref{P4} applied with $K=\Gamma$. 
After fixing the $f_i$'s, let us choose numbers 
$0 < \epsilon_0 < 1$ and $C_0 > 1$  in such a way that 
\begin{equation}
\label{eq:7}
 		\sup_\Gamma \left| \exp\left( 
 		\sum_{i=1}^{2g} \tau_i f_i\right) - 1 -  
 		\sum_{i=1}^{2g} \tau_i f_i \right|  
 		\le C_0 \max_i |\tau_i|^2   
\end{equation}
holds if $\tau_i\in\C$ and $\max_i |\tau_i| \le \epsilon_0$.     

Let $c=\max_i |c_i|$. By decreasing the number 
$\epsilon_0>0$ if necessary we can assume that
\[
	8\pi C_0 e^{1+c} \epsilon_0 <1.
\]
Choose a constant  $C_1>0$ such that
\[
	|e^t - 1| <C_1 |t| \quad \hbox{if}\ |t| < \epsilon_0.
\]
Then, by (\ref{eq:6}) and (\ref{eq:7}), 
it is easy to see that, for any positive number $\epsilon>0$ 
satisfying  
\[
	8\pi C_1 \left (1+\sup_\Gamma \sum_{i=1}^{2g} 
	|f_i|\right)\epsilon  <\epsilon_0
\]
and for any $h\in \cO_{alg}(R)$ satisfying (\ref{eq:4}), 
the inequality  
\[
		\left| \tau_j  - \int_{\alpha_j} 
		\exp\left( \sum \tau_i f_i -h\right) \omega \right| 
		\le \frac{\epsilon_0}{2} 
\]
holds for every $j=1,\ldots, 2g$ whenever $\max_i |\tau_i| \le \epsilon_0$.     
Hence, for such a choice of $h$, the map  
\[
	\C^{2g} \ni \tau=(\tau_1,\ldots,\tau_{2g}) 
	\stackrel{\Phi}{\longrightarrow}
	\left(\Phi_1(\tau),\ldots,\Phi_{2g}(\tau)\right) \in\C^{2g},              
\]
whose $j$-th component is defined by 
\[
	\Phi_j(\tau) = \int_{\alpha_j} 
	\exp\left( \sum_{i=1}^{2g} \tau_i f_i - h\right) \omega,
\]
maps the polydisc 
$P=\{\tau \in\C^{2g} \colon \max|\tau_i| < \epsilon_0\}$
onto a neighborhood of the origin in $\C^{2g}$.
In particular, we have $\Phi(\zeta)=0$ 
for some point $\zeta=(\zeta_1,\ldots,\zeta_{2g}) \in P$,
and for this $\zeta$ the equations (\ref{eq:5}) hold. 
This concludes the proof of Theorem \ref{T1}.

\section{Concluding remarks}
By a minor adjustment of the proof of Theorem \ref{T1}
one can construct a nowhere vanishing holomorphic 
1-form of finite order, $\omega$, on $R$ whose periods 
$\int_{\alpha_j} \omega$ over the basis curves 
$[\alpha_j]$ of $H_1(R;\Z)$ are arbitrary given 
complex numbers. In other words, one can prove 
the following result.
(See Kusunoki and Sainouchi \cite{KusSai} and 
Majcen \cite{Majcen} 
for the corresponding result on open Riemann surface 
and without the finite order condition.)

\begin{theorem}
Let $X$ be a compact Riemann surface and $x_0\in X$.
Every element of the de Rham cohomology group $H^1(X;\C)$ is represented
by a nowhere vanishing holomorphic 1-form of finite order
on $R=X\bs\{x_0\}$. 
\end{theorem}

Since every affine algebraic curve $A\subset \C^N$
is obtained by deleting finitely many points
from a compact Riemann surface, 
Theorem \ref{T1} implies that 
{\em every affine algebraic curve admits a 
noncritical holomorphic function of finite order.}
One may ask whether the same result also holds 
on higher dimensional algebraic manifolds:

\begin{problem}
Does every affine algebraic manifold $A\subset \C^N$
of dimension $\dim A>1$ admit a 
noncritical holomorphic function $f\colon A\to\C$
of finite order?
\end{problem}

Here we say that $f$ is of finite order if
$|f(z)|\le \lambda\exp{|z|^\mu}$ holds for all 
$z\in A$ and for some pair of constants $\lambda,\mu>0$.

Since such $A$ is a Stein manifold, it admits a noncritical
holomorphic function according to \cite{Forstneric}. 
The construction in that paper 
is quite different from the one presented
here even for Riemann surfaces, and it
does not necessarily give a function of finite order
when $A$ is algebraic. The main difficulty 
is that the closedness
equation $\d \omega = 0$ for a holomorphic 1-form, 
which is automatically satisfied on a Riemann surface,
becomes a nontrivial condition when $\dim A >1$.
In particular, this condition is not preserved 
under multiplication by holomorphic functions,
and hence one can not hope to adjust the periods
in the same way as was done above.

\end{document}